\documentclass[12pt]{amsart}

\usepackage{amssymb,amsfonts,amsthm,amsmath,amscd}
\usepackage{tikz}
\usepackage{fancybox}
\usepackage{pgf,pgfarrows,pgfnodes,pgfautomata,pgfheaps,pgfshade}
\usepackage{graphicx}
\usepackage[all,graph,frame]{xy}
\usepackage{appendix}
\usepackage{enumerate}

\newtheorem{proposition}{Proposition}
\newtheorem{lemma}[proposition]{Lemma}
\newtheorem{theorem}[proposition]{Theorem}
\newtheorem{corollary}[proposition]{Corollary}

\theoremstyle{definition}

\newtheorem{remark}[proposition]{Remark}

\def\G{\mathcal{G}}
\def\P{\mathcal{P}}

\def\Z{\mathbb{Z}}

\begin{document}

\title{{\sc Slow k-Nim}}


\author{Vladimir Gurvich}
\address{MSIS and RUTCOR, RBS, Rutgers University,
100 Rockafeller Road, Piscataway, NJ 08854}
\email{vladimir.gurvich@rutgers.edu}

\author{Nhan Bao Ho}
\address{Department of Mathematics and Statistics, La Trobe University, Melbourne, Australia 3086}
\email{nhan.ho@latrobe.edu.au, nhanbaoho@gmail.com}


\begin{abstract}
Given  $n$  piles of tokens and a positive integer  $k \leq n$, we study the following two impartial combinatorial games {\sc Nim}$^1_{n, \leq k}$  and  {\sc Nim}$^1_{n, =k}$. In the first (resp.~second) game, a player, by one move, chooses at least $1$ and at most (resp.~exactly) $k$ non-empty piles and removes one token from each of these piles. For the normal and mis\`ere version of each game we compute the Sprague-Grundy function for the cases $n = k = 2$  and  $n = k+1 = 3$. For game {\sc Nim}$^1_{n, \leq k}$ we also characterize its $\P$-positions for the cases  $n \leq k+2$  and  $n = k+3 \leq 6$.
\end{abstract}

\maketitle


\section{Introduction and previous results}
\label{s0}

We assume that the reader is familiar with basics
of the Sprague-Grundy (SG) theory of impartial games \cite{Gru39, GS56, Spr35, Spr37}.
In this paper we will need only the concept of
the SG function and $\P$-positions for the normal and mis\`ere versions, which
are presented in almost every paper about impartial games;
see, for example, \cite{BGHM15, BGHMM15, GurHo15}, which are closely related to the present paper.
An introduction to the SG theory
can be found in \cite{BCG01-04, Con76}.

We denote by  $\Z_\geq$ (resp.~$\Z_>$) the set of non-negative
(resp.~positive) integers. For $t \in \Z_\geq$, a position whose SG value is $t$  will be called a $t$-position.

In this section we briefly survey
several variants of the game  {\sc Nim}.
The normal version is considered in Sections 1.1 - 1.4, and the mis\`ere one in Section \ref{Ss-mis}.


\subsection{Classic {\sc Nim}}  \label{Ss.Nim}
The ancient game of {\sc Nim} is played as follows.
There are   $n$  piles containing  $x_1, \ldots, x_n$  tokens.
Two players alternate turns.
By one move, a player chooses a non-empty pile  $x_i$ and removes
an arbitrary (strictly positive) number of tokens from it.
The game terminates when all piles are empty.
The player who made the last move wins the normal version of the game and loses its mis\`ere version.
Both versions were solved
(all $\P$-positions found) by Charles Bouton in 1901 in his seminal paper \cite{Bou901}.

By definition, an $n$-pile {\sc Nim} is the sum of $n$ one-pile {\sc Nim}s
and hence the SG function of position $x = (x_1, \ldots x_n)$  is the binary bitwise sum
(so-called Nim-sum) \cite{Bou901, Gru39, Spr35, Spr37}
\begin{equation*} \label{eq0}
\G(x) = \G(x_1) \oplus \cdots \oplus \G(x_n) = x_1  \oplus \cdots \oplus x_n.
\end{equation*}
$\P$-positions of {\sc Nim}, as well as of any impartial game, are
the zeros of its SG function, that is,  $x$  is a $\P$-position if and only if  $\G(x) = 0$.


\subsection{Moore's $k$-{\sc Nim}} \label{Ss.Moore}
In 1910, Eliakim Hastings Moore \cite{Moo910} suggested the following generalization of
{\sc Nim} for any $k,n \in \Z_>$  such that  $k \leq n$.
In this game a player,  by one move, reduces at least one and at most  $k$  piles.
The game turns into the standard  $n$-piles  {\sc Nim}  when  $k=1$ and
is trivial when $k=n$; in the latter case it turns into
{\sc Nim} with one pile of size $s = \sum_{i=1}^n x_i$.
Moore denoted his game by {\sc Nim$_k$};
we will call it Moore's $k$-{\sc Nim} and denote by {\sc Nim}$_{n, \leq k}$.

Moore gave a simple and elegant formula characterizing of the  $\P$-positions of {\sc Nim}$_{n, \leq k}$,
thus, generalizing Bouton's results.
Later, Berge \cite{Ber62}, Jenkyns and Mayberry \cite{JM80}  tried to extend
Moore's formula for the SG function of {\sc Nim}$_{n, \leq k}$ as follows.

Present the cardinalities of  $n$  piles as the binary numbers,
$x_i = x_{i_0} + x_{i_1}2 + x_{i_2}2^2 + \cdots + x_{i_m}2^m$,
take their bitwise sum modulo $k+1$ for every bit  $j$,
that is,  set  $y_j = \sum _{i=1}^n x_{i_j} \bmod{(k+1)}$,
and denote the obtained number  $y_0 y_1 \ldots y_m$ in base $k+1$  by $M(x)$.

In \cite{Moo910}  Moore claimed  that
$x$  is a $\P$-position of {\sc Nim}$_{n, \leq k}$  if and only if  $M(x) = 0$.
This is a generalization of Bouton's result \cite{Bou901} corresponding to the case  $k=1$.
Indeed, in this case  $M(x)$  turns into  $\G(x)$  and
{\sc Nim}$_{n, \leq k}$  turns into the standard $n$-pile  {\sc Nim}.

\begin{remark}
\label{SG-Moore}
Moore published his results in 1910, while the SG function
was introduced only quarter of century later, in 1935-9 \cite{Gru39, Spr35, Spr37}.
It was shown in these papers that the case  $k=1$  has the following remarkable property.
The SG function can be used to solve the sum of arbitrary  $n$  impartial games
$($in the normal version;
not only the $n$-pile {\sc Nim}, which is the sum of  $n$  one-pile {\sc Nim} games$)$.
The concept of the sum can be generalized to that of the $k$-sum as follows:
by one move a  player chooses at least one and at most  $k$  game-summands and
makes an arbitrary legal move in each of them.
For example, {\sc Nim}$_{n, \leq k}$  is the $k$-sum of $n$ one-pile {\sc Nim} games.
By Moore's result, function $M_{n,k}(x)$ can be used to find all  $\P$-positions of the latter, but
this application works only for the standard sum,
which corresponds to the case  $k=1$; see Section 2.4 of \cite{BGHMM15} for more details.
\end{remark}

Moore's result for  $\P$-positions was extended to $1$-positions
by Jenkyns and Mayberry \cite{JM80}; see also \cite{BGHM15} for an alternative proof.

\begin{proposition} [\cite{JM80}]
\label{p1}
For  $t = 0$  and  $t = 0$, vector  $x$  is
a $t$-position of {\sc Nim}$_{n, \leq k}$
if and only if  $M(x) = t$.
\qed
\end{proposition}

In his book \cite{Ber62}, Claude Berge claimed that Proposition \ref{p1}
can be generalized for any  $t$, or in other words, that
$\G(x) = M(x)$; see \cite[page~55, Theorem 3]{Ber62}.
However, this is an overstatement, as it was pointed out in  \cite{JM80}.
The claim only holds when $M(x) \leq 1$;
see Theorems 11 and 12 on page 61 and also page 53 of \cite{JM80}.
For  $t > 1$,  the  $t$-positions are no longer related to Moore's function  $M$;
moreover, it seems difficult to characterize them in general, for all  $k$;
see \cite{BGHM15, BGHMM15} for more details.
Yet, the case $n = k+1$  is tractable:
for $n = 2$ the game turns into the two-pile {\sc Nim};
for $n > 2$ the SG function was first given in \cite{JM80};
this result was rediscovered and generalized in \cite{BGHM15} as follows.


\subsection{{Game \sc Exco-Nim}, as a generalization of Moore's Nim with  $n = k+1$}
\label{Ss.Exco}
The following game called {\sc Exco-Nim}
(extended complementary {\sc Nim})  was suggested in \cite{BGHM15}.
Given $n+1$  piles containing
$x_0, x_1, \ldots, x_n$  tokens, two players move alternatingly.
By one move, a player can reduce  $x_0$
together with at most  $n-1$  of the remaining  $n$  numbers such that
at least one token is removed.

Obviously, this game is a generalization of Moore's {\sc Nim} with  $k = n-1$:
the former turns into the latter when  $x_0 = 0$.
In case  $n \geq 3$, it is not difficult to generalize
Jenkyns and Mayberry's formula for the SG function of Moore's {\sc Nim} as follows.

\begin{theorem} [\cite{BGHM15}] \label{t1}
Given a position  $x = (x_1, \ldots, x_n)$, let us define
\begin{equation}
\label{eq2}
u(x) = \sum_{i = 0}^n x_i, \quad  m(x) = \min_{i = 1}^n x_i,
\quad y = u(x) - n \, m(x),  \quad z = \frac{1}{2} (y^2 + y + 2).
\end{equation}
Then, for any  $n \geq 3$ the SG function
of {\sc Nim}$_{n, \leq n-1}$  is given by the formula:
\begin{equation} \label{eq3}
\G(x) ~=~
\begin{cases}
u(x),                       & \text{ if } m(x) < z;\\
(z-1) + (m-z) \bmod{(y+1)}, &\text{ otherwise.}
\end{cases}
\end{equation}
\qed
\end{theorem}

In case $n=2$  Moore's {\sc Nim}  turns into the
standard  {\sc Nim} with two piles, which is trivial.
Somewhat surprisingly, {\sc Exco-Nim}
in this case becomes very difficult;
see \cite{BGHM15} for partial results.

\subsection{Exact {\sc $k$-Nim}}
\label{Ss.Exact}
In \cite{BGHMM15}, one more version of {\sc Nim} was introduced;
it was called {\sc Exact $k$-{\sc Nim}} and denoted {\sc Nim}$_{n, =k}$.
Given integer $k, n$ such that  $0 < k \leq n$  and  $n$ piles of $x_1, \ldots, x_n$ tokens,
by one move a player chooses {\em exactly}  $k$ piles and
removes an arbitrary (strictly positive) number of tokens from each of them.
Clearly, both games, {\sc Nim}$_{n, =k}$  and {\sc Nim}$_{n, \leq k}$,
turn into the standard  $n$-pile {\sc Nim} when  $k=1$  and become trivial when $k=n$.
In \cite{BGHMM15}, the SG function $\G(x)$  of {\sc Nim}$_{n, =k}$  was
efficiently computed for the case  $n \leq 2k$ as follows.

The {\em tetris} function
$T_{n,k}(x)$  was defined in \cite{BGHMM15} by the formula:
\begin{align} \label{tetris}
T_{n,k}(x) = \max\{m \, | \, \exists v_i = (v^i_1, \ldots, v^i_n), 1 \leq i \leq m, \;
v^i_j \in \{0, 1\}, \; \sum_{i = 1}^m v^i_j \leq x_j\}.
\end{align}

In other words, $T_{n,k}(x)$  is the largest number
$m$ of $n$-vectors
$(v^i_1, \ldots, v^i_n)$ with exactly $k$  coordinates 
$1$ and $n-k$ entries $0$ such that
each sum $\sum_{i = 1}^m v^i_j$ of the $j$th coordinates does not exceed $x_j$.
For example, $T_{4,3}(1,1,2,3) = 2$ since $v_1 = (0,1,1,1)$ and $v_2 = (1,0,1,1)$ satisfy (\ref{tetris})
and any three $4$-vectors with exactly three $1$ and one $0$
result in $\sum_{i = 1}^3 v^i_j \geq 2$ for more than two coordinates.

In \cite{BGHMM15}, an algorithm computing  $T_{n,k}(x)$ in polynomial time was obtained and
the following statement shown.

\begin{proposition} [\cite{BGHMM15}]
\label{p2}
The SG function $\G$  of the game {\sc Nim}$_{n, =k}$ and
the tetris function $T_{n,k}$  are equal whenever $2k > n$.
\qed
\end{proposition}

Another tractable (but more difficult) case is  $n = 2k$.
Somewhat surprisingly, the games {\sc Nim}$_{n, = k}$ for $n = 2k$  and
{\sc Nim}$_{n, \leq k}$  for $n = k+1$ appear to be very similar.

\begin{theorem} [\cite{BGHMM15}]
\label{t2}
For  $n = 2k \geq 4$, the SG function $\G(x)$  of
{\sc Nim}$_{n, = k}$ is still given by formula $(\ref{eq3})$,
but the variables  $u = u(x)$   and  $y = y(x)$  in it are defined
by means of the tetris function
$($rather than by $(\ref{eq2}$$)$  as follows:

\smallskip
\noindent
$u(x) = T_{n,k}(x)$,  $y(x) = T_{n,k}(x')$  for
$x' = (x_1 - m(x), \ldots, x_n - m(x))$.
\qed
\end{theorem}

The game becomes even more difficult when  $n > 2k$.
No closed formula is known already for the $\P$-positions of {\sc Nim}$_{5, = 2}$.

\subsection{Tame, pet, and miserable games} \label{Ss-mis}
In Sections \ref{Ss.Nim} - \ref{Ss.Exact} above, we surveyed several variants of {\sc Nim},
but considered only their normal versions.
Now we will recall some results for the mis\`{e}re play.
A position $x$ from which there is no legal move is called {\em terminal}.
The SG functions  $\G$  and  $\G^-$ of the normal and mis\`{e}re versions
are defined by the same standard SG recursion, but are initialized differently:
$\G(x) = 0$, while  $\G^-(x) = 1$, for any terminal position  $x$.

A position  $x$  is called an  $i$-{\em position}
(resp.~$(i,j)$-{\em position}) if $\G(x) = i$ (resp.~$\G(x) = i$  and  $\G^-(x) = j$).
We will denote by $X_T$,  $X_i$, and  $X_{i,j}$  the sets of all terminals, $i$-,
and $(i,j)$-positions, respectively.
By definition, every terminal position is a $(0,1)$-positions,
$X_T \subseteq X_{0,1}$.

In 1976 John Conway \cite{Con76} introduced the so-called {\em tame} games, which
contain only  $(0,1)$-, $(1,0)$- and  $(t,t)$-positions.
The positions of the first two classes are called the {\em swap} positions.

A game is called {\em miserable} \cite{Gur07, Gur12} if
from each $(1,0)$-position there is a move to a $(0,1)$-position,
from each non-terminal $(0,1)$-position there is a move to a $(1,0)$-position, and
finally, from every non-swap position, either
there is a move to a $(0,1)$-position and
another one to a $(1,0)$-position, or
there is no move to a swap position at all.

\begin{proposition} [\cite{Gur07, Gur12}]
\label{gur07}
Miserable games are tame.
\qed
\end{proposition}

In \cite{Gur07} this result was applied to the game Euclid.
Many other applications are suggested in \cite{GurHo15}.
It appears to be also applicable to the Moore's {\sc Nim}.

\begin{proposition} [\cite{GurHo15}] \label{p4} 
Game {\sc Nim}$_{n, \leq k}$  is miserable.
Moreover, let $x = (x_1, \ldots, x_n)$ be a position in {\sc Nim}$_{n, \leq k}$
such that  for $2 \leq k < n$ and
let $l = l(x)$ be the number of non-empty piles in $x$. Then
\begin{itemize}  \itemsep0em
\item [\rm{(i)}]  $x \in X_{0,1}$ if and only if $x_i \leq 1$ for all $i$ and $l \equiv 0 \mod{(k+1)}$;
\item [\rm{(ii)}] $x \in X_{1,0}$ if and only if $x_i \leq 1$ for all $i$ and $l \equiv 1 \mod{(k+1)}$.
\end{itemize} \qed
\end{proposition}

The normal and mis\`ere versions of the game {\sc Exact $k$-Nim} are related
in accordance with the following two statements.

Let $d(x)$ is the largest number of moves from $x$ to the terminal position.

\begin{proposition}  [\cite{BGHMM15}] \label{p5}
The game of {\sc Nim}$_{2k, = k}$ is miserable. Moreover, $x = (x _1, x_2, \ldots, x_n)$ is
\begin{itemize} \itemsep0em
\item  [\rm{(i)}]  a $(0,1)$-position if and only if $x_1 = x_2 = \cdots = x_{k+1} \leq 1$;
\item  [\rm{(ii)}] a $(1,0)$-position if and only if Tetris function $T_{2k,k}(x) = 1$.
\end{itemize}
\end{proposition}

Let us remark that {\sc Nim}$_{n, = k}$  may be not tame
(and, hence, not miserable) when  $2k < n$, which
indicates indirectly that the game becomes much more difficult in this case.
For example, our computations show that $(1,2,3,3,3)$  is a  $(0,2)$-position of {\sc Nim}$(5, = 2)$.

\smallskip

In contrast, the case  $2k > n$  is much simpler;
in this case game {\sc Nim}$_{n, = k}$  satisfies even a stronger property.

In \cite{Gur12}, an impartial game was called {\em pet} if
it contains only $(0,1)$-, $(1,0)$-, and $(t,t)$- positions with  $t \geq 2$.
Recall that only  $t \geq 0$  is required for the tame games.
Hence, any pet game is tame.
Several characterizations of the pet games were given in \cite{Gur12}.
In particular, it was shown that the following properties of an impartial game are equivalent:

\smallskip

\rm{(i)} the game is pet,

\rm{(ii)}  there are no $(0,0)$-positions,

\rm{(iii)} there are no $(1,1)$-positions,

\rm{(iv)} from every non-terminal $0$-position there is a move to a $1$-position.

\smallskip

The last property was considered already in 1974
by Thomas Ferguson \cite{Fer74}, for a different purpose.
He proved that  (iv)  holds for the so-called subtraction games;
see \cite{Fer74} and also \cite{Gur12, GurHo15}
for the definition, proof, and more details.

\begin{proposition} [\cite{BGHMM15, GurHo15}]
\label{p6}
The game {\sc Nim}$_{n, = k}$  is pet whenever $n < 2k$.
\qed
\end{proposition}

Let us note that Moore's game {\sc Nim}$_{n, \leq k}$ is pet only when
$k = n$  \cite{BGHM15, BGHMM15}.
In this case the game is trivial and equivalent to
the standard  {\sc Nim}  with only one pile.


\section{Two versions of {\sc Slow $k$-Nim}} \label{s1}

In this paper, we introduce two games modifying
{\sc Nim}$_{n, \leq k}$  and {\sc Nim}$_{n, = k}$, respectively.
We keep all their rules, but add the following extra restriction:
by one move any pile can be reduced by at most one token.
We will call the obtained two games
by {\sc Slow Moore's $k-$Nim} and {\sc Slow Exact $k-$Nim}  and denote them by
{\sc Nim}$^1_{n, \leq k}$  and  {\sc Nim}$^1_{n, = k}$, respectively.
As usual, the player who makes the last move wins
the normal and loses mis\`{e}re version.

A position is a non-negative  $n$-vector $x =(x_1, \ldots, x_n)$, as before.
Without loss of generality, we assume that the coordinates are not decreasing, that is,
$0 \leq x_1 \leq \cdots \leq x_n$.
Yet, after a move  $x \to x'$
(from  $x$  to  $x'$) this condition may fail for  $x'$.
In this case we will reorder the coordinates of  $x'$ to reinforce it.

When  $k=1$, the games {\sc Nim}$^1_{n, \leq k}$ and {\sc Nim}$^1_{n, = k}$ are identical and trivial.
The SG function depends only on the parity of the total number $s = \sum_{i=1}^n x_i$  of tokens:
if  $s$  is even, $\G(x) = 0$, $\G^-(x) = 1$, and vice versa if  $s$  is odd.

Also {\sc Nim}$^1_{n, = k}$  is trivial when  $k = n$.
In this case any play consists of  $m = x_1$  
moves and hence the SG function depends only on the parity of  $m$:
if  $m$  is even, $\G(x) = 0$, $\G^-(x) = 1$, and vice versa if  $m$  is odd.

Thus, in all trivial cases considered above, both games are pet;
moreover, they have only  $(0,1)$- and  $(1,0)$-positions that
alternate with every move.

In Section \ref{S.Moore} we analyze the SG function of {\sc Slow Moore's Nim}
for  $n = 2$  and  $n =3$ and show that in these cases $\P$-positions
of the game are characterized by the parities of its coordinates. 
Then, we discuss miserability.
In Section \ref{S.Exact} we study  {\sc Slow Exact $k-$Nim} for $n = 3$ and
obtain closed formulas for the SG functions for both normal and mis\`{e}re version of Nim$^1_{3,\leq 2}$.

\medskip

Given a non-negative integer $n$-vector  $x = (x_1, \ldots, x_n)$  such that
$x_1 \leq \cdots \leq x_n$, its {\em parity vector} $p(x)$  is defined as
$n$-vector $p(x) = (p(x_1), \ldots, p(x_n))$ whose coordinates take values $\{e,o\}$
according to the natural rule:
$p(x_i) = e$  if  $x_i$  is even and $p(x_i) = o$  if  $x_i$  is odd.


\section{{\sc Slow Moore's $k-$Nim}} 
\label{S.Moore}

The status of a position  $x$  in {\sc Nim}$^1_{n, \leq k}$ is often uniquely determined
by its parity vector $p(x)$.
In particular, this is the case for  $n = k = 2$  and  $n = k+1 = 3$.
First, we will prove this claim and then characterize
the  $\P$-positions for several other cases.

\subsection{Sprague-Grundy values for $n = 2$  and $n = 3$}

\begin{proposition}  \label{p7}
The SG values  $\G(x)$   for the cases
$n = k = 2$ and $n = 3, k = 2$  are uniquely defined by
the parity vector  $p(x)$  as follows:
\begin{enumerate}
\item [\rm{(i)}]   For $n = k = 2$,
    \begin{align*}
    \G(x) =
    \begin{cases}
    0, \text{ if } p(x)  = (e,e); \\
    1, \text{ if } p(x)  = (e,o); \\
    2, \text{ if } p(x)  = (o,o); \\
    3, \text{ if } p(x)  = (o,e).
    \end{cases}
    \end{align*}
\item [\rm{(ii)}]  For $n = 3, k = 2$,
    \begin{align*}
    \G(x) =
    \begin{cases}
    0, \text{ if } p(x)  \in \{(e,e,e), (o,o,o)\}; \\
    1, \text{ if } p(x)  \in \{(e,e,o), (o,o,e)\}; \\
    2, \text{ if } p(x)  \in \{(e,o,o), (o,e,e)\}; \\
    3, \text{ if } p(x)  \in \{(e,o,e), (o,e,o)\}.
    \end{cases}
    \end{align*}
\end{enumerate}
\end{proposition}
\proof
We will prove only \rm{(ii)}. The proof of \rm{(i)} is similar, simpler,
and we leave it to the reader.

Let $S_0$ (resp.~$S_1$, $S_2$, $S_3$) be the set of positions whose parity vectors belong to

\smallskip
\noindent
$\{(e,e,e), (o,o,o)\}$ (resp.~$\{(e,e,o), (o,o,e)\}$, $\{(e,o,o), (o,e,e)\}$, $\{(e,o,e), (o,e,o)\}$.

\smallskip

Obviously, sets  $S_g, g \in \{0,1,2,3\}$, partition all positions of the game, since
there exist exactly  $2^3 = 8$  parity vectors and all are listed above.

\smallskip

Recall that the set of  $g$-position is defined by the following two claims:

\begin{enumerate}
\item  there exists no move between positions of $S_g$, for any fixed  $g \in \{0,1,2,3\}$;
\item  for any  $i,g \in \{0,1,2,3\}$ such that  $i < g$  and  $x \in S_g$, there is a move from $x$ to $S_i$,
in other words, $S_i$  is reachable from  $S_g$.
\end{enumerate}


Let us begin with (1). It is easily  seen that
a move from  $x$ to $y$  exists only if
these two positions have either of the following two properties:


\begin{itemize} \itemsep0em
\item [\rm{(a)}]  the parities in one of their three coordinates are the same, while the parities in the two other are opposite;
then,  $y$  might be reachable from  $x$  by reduction of these two ``other" piles; conversely,
\item [\rm{(b)}]  the parities in one of their three coordinates are opposite, while the parities in the remaining one are the say;
then,  $y$  might be reachable from  $x$  by reduction of these one ``remaining" pile.
\end{itemize}

It is easy to verify that
neither \rm{(a)} nor \rm{(b)} can hold for any two positions
that belong to the same set  $S_g$  for any fixed   $g \in \{0,1,2,3\}$.
Thus, $(1)$ holds.

\medskip

To show (2), we consider four cases assuming that
$x = (x_1, x_2, x_3) \in S_g$  for  $g \in \{0,1,2,3\}$.

If  $g = 1$, it is sufficient to reduce $x_3$  by one to get $y \in S_0$.

Let  $g = 2$.
If  $p(x) = (e, o, o)$, reduce $x_2$ and $x_3$  to get  $y \in S_0$ with $p(y) = (e,e,e)$ and
reduce  $x_2$ to get  $y \in S_1$  with $p(y) = (e,e,o)$.
If $p(x) = (o, e, e)$, reduce $x_1$ to get $y \in S_0$ with $p(x') = (e,e,e)$ and
reduce  $x_2$  to get  $y \in S_1$  with  $p(y) = (o,o,e)$.
Recall that the coordinates of a vector are assumed to be not decreasing.

\smallskip

Finally, let  $g = 3$.
If $p(x) = (e,o,e)$, then
reduce  $x_2$  to get  $y \in S_0$  with $p(y) = (e,e,e)$,
reduce $x_2$ and $x_3$  to get $y \in S_1$ with $p(y) = (e,e,o)$, and
reduce $x_3$  to get  $y \in S_2$ with $p(y) = (e,o,o)$.

If $p(x) = (o,e,o)$, then
reduce $x_2$  to get $y \in S_0$ with $p(y) = (o,o,o)$,
reduce $x_1$ to get  $y \in S_1$ with $p(x') = (e,e,o)$,
and reduce $x_1$ and $x_2$  to get  $y \in S_2$ with $p(y) = (e,o,o)$.
\qed

\subsection{$\P$-position for the cases $n \leq k+2$ and $n = k+3 \leq 6$}

In both these cases the $\P$-positions again are
simply characterized by the corresponding parity vectors.

\begin{proposition} \label{p8}
The $\P$-positions of game {\sc Nim}$^1_{n, \leq k}$ are uniquely defined
by their parity vectors as follows:  
\begin{enumerate} \itemsep0em
\item [$(1)$] for $n=k$ we have: $x \in \P$ if and only if $p(x) = (e, e, \ldots, e);$
\item [$(2)$] for $n=k+1$ we have: $x \in \P$ if and only if $p(x) \in \{(e, e, \ldots, e), (o, o, \ldots, o)\};$
\item [$(3)$] for $n=k+2$ we have: $x \in \P$ if and only if $p(x) \in \{(e, e, \ldots, e), (e,o \ldots, o)\};$
\item [$(4)$] for $n=5, k=2$ we have: $x \in \P$ if and only if

\smallskip
\noindent
$p(x) \in \{(e,e,e,e,e), (e,e,o,o,o), (o,o,e,e,o), (o,o,o,o,e)\};$
\item [$(5)$] for $n=6, k=3$ we have: $x \in \P$ if and only if

\smallskip
\noindent
$p(x) \in \{(e,e,e,e,e,e), (e,e,o,o,o,o), (o,o,e,e,o,o), (o,o,o,o,e,e)\}.$
\end{enumerate}
\end{proposition}

\proof
For each case, we have to verify the following two statements:

\begin{enumerate} \itemsep0em
\item [\rm{(i)}]   there is no move from  $x$  to  $y$
if both  positions $x$ and $y$  are in $\P$, and
\item  [\rm{(ii)}] for any $x \notin \P$, there is a move to a position $y \in \P$.
\end{enumerate}

Property \rm{(i)} is obvious for all five cases.
Indeed, it is enough to notice that
in each case, the Hamming distance between 
$p(x)$ and $p(y)$  is either  $0$
(if they coincide) or larger than the corresponding  $k$
(if they are distinct).
In both cases, there is no move from $x$ to $y$. 

\medskip

To verify \rm{(ii)}, consider five statements $(1) - (5)$ separately.
In each case, for every position  $x \not \in X$, we will construct a move to a position  $y \in \P$.

\medskip

To show  $(1)$  notice that
$x$ has odd coordinates whenever  $x \notin \P$.
Then, to move from  $x$ to a position $y \in \P$, it
is sufficient to reduce all odd coordinates of  $x$.

\medskip

To show $(2)$ notice that  $x$ has at least one even and at least one odd coordinate whenever $x \notin \P$.
If $x_1 = 0$, we reduce all odd coordinates 
to get a position $y$ with $p(y) = (e,e,\ldots,e)$.
If $x_1 > 0$, we reduce 
each coordinate whose parity is different from $p(x_1)$, to get a position $y$
whose all coordinates have the same parities as
$p(x_1)$, that is, either $(e, e, \ldots, e)$ or $(o, o, \ldots, o)$.

\medskip

Let us show $(3)$.
If all coordinates of $x$ are  odd, reduce $x_1$, to get a position $y$ with $p(y) = (e,o,o,\ldots,o)$.
Otherwise, let us consider two cases.
Let exactly one coordinate of  $x$ be  even.
It cannot be  $x_1$, since $x \notin \P$.
Reduce this even coordinate together with $x_1$ to get a position  $y$  with the parity vector $(e,o,o, \ldots,o)$.
If at least two coordinates of $x$  are even, reduce all odd coordinates  
to get a position  $y$  with the parity vector $(e,e,\ldots,e)$.

\medskip

To show  $(4)$ and $(5)$, let us consider the following cases.

If all coordinates of $x$ are odd, reduce $x_1$ and $x_2$  to get a position $y$ with $p(y) = (e,e,o,\ldots,o)$.

If exactly one coordinate  $x_i$  in  $x$  is even then it is easily seen that
$i \neq 5$ for the case $(4)$ and $i \neq 6$ for the case $(5)$.

\begin{itemize} \itemsep0em
\item If $i$ is odd, reduce $x_{i+1}$  
to get a position $y$ with \\
$p(y) \in \{(e,e,o,o,o), (o,o,e,e,o)\}$ for the case $(4)$ and \\
$p(y) \in \{(e,e,o,o,o,o), (o,o,e,e,o,o), (o,o,o,o,e,e)\}$ for the case $(5)$.
\newline
Recall that  $x_{i+1} \geq x_i+1$ and note that
the order of these two coordinates does not change after the above move.
\item Let $i$  be even.
If $i = 2$, reduce $x_1$ to get a position $y$ with
$p(y) = (e,e,o,o,o)$ for the case $(4)$ and $p(y) = (e,e,o,o,o,o)$  for the case $(5)$.
If $i = 4$, reduce the last $2$
(resp.~$3$) coordinates $x_4, x_5$
(resp.~$x_4, x_5, x_6$) to get a position $y$ with $p(y) = (o,o,o,o,e)$
(resp.~ $p(y) = (o,o,o,o,e,e)$) for case  $(4)$ (resp.~$(5)$).
\end{itemize}

If at least $3$ coordinates of  $x$  are even, reduce all odd coordinates to get a position  $y$ with $p(y) = (e,e,\ldots,e)$.

\medskip

It remains to consider the case when there are exactly two even coordinates $x_i, x_j$, $i \neq j$.
Note that $(i,j) \notin \{(1,2), (3,4), (5,6)\}$.
If $i$ is odd, reduce $x_{i+1}$ and $x_j$  to get a position $y$ with
$p(y) \in \{(e,e,o,o,o), (o,o,e,e,o)\}$ for the case $(4)$ and
$p(y) \in \{(e,e,o,o,o,o), (o,o,e,e,o,o), (o,o,o,o,e,e)\}$ for the case $(5)$.
If $i = 2$, reduce $x_i$ and $x_j$   
to get a position $y$ with $p(y) = (e,e,o,o,o)$
(resp.~$p(y) = (e,e,o,o,o,o)$) for the case $(4)$ (resp.~$(5)$).
If $i = 4$, reduce $x_i$ for case $(4)$ to get a position $y$ with $p(y) = (o,o,o,o,e)$, while
for case $(5)$, reduce 
$x_i$ and the unique odd coordinate of the last two, $x_5$ and $x_6$.
(one of them is $x_j$ which is even)
to get a position $y$ with $p(y) = (o,o,o,o,e,e)$.
\qed

\bigskip
\begin{sloppypar}
Our computations show that  Proposition \ref{p8}  cannot be extended to cover the case  $n \leq 6$, because
in the remaining subcases the  $\P$-positions are not uniquely characterized in terms of their parity vectors.
For example, if  $n = 6$ and $k = 2$  then  $(3,3,3,4,4,4)$  is a $\P$-position, while $(1,1,1,2,2,4)$  is a $2$-position,
although their parity vectors are the same;
also $(1, 3, 3, 3, 3, 3)$, $(1, 3, 3, 3, 5, 5)$, and  $(1, 3, 5, 5, 5, 5)$  are $\P$-positions,
while $(1, 1, 1, 1, 1, 3)$, $(1, 1, 1, 1, 1, 5)$ are $2$-positions, and $(1, 1, 3, 3, 5, 5)$  is a $3$-position.
\end{sloppypar}

\subsection{Miserability}
Now let us consider the mis\`ere version of  {\sc Nim}$^1_{n, \leq k}$.

\begin{proposition} \label{p9}
When $k \geq n-1$,  {\sc Slow Moore's Nim}$^1_{n, \leq k}$ is miserable. Furthermore,

\medskip
\noindent
$V_{0,1} = \{(0,0,\ldots,0,2j) | j \in \Z_{\geq}\}$ and $V_{1,0} = \{(0,0,\ldots,0,2j+1) | j \in \Z_{\geq}\}$ for $k = n$;

\medskip
\noindent
$V_{0,1} = \{(i,i,\ldots,i,i+2j) | i, j \in \Z_{\geq}\}$ and $V_{1,0} = \{(i,i,\ldots,i,i+2j+1) | i, j \in \Z_{\geq}\}$ for  $k = n-1$.
\end{proposition}

The proof uses the following lemma from \cite{GurHo15} characterizing miserability.

Given a game  $G$  and two sets  $V', V''$ of its positions, we say that
$V'$ is {\em movable} to  $V''$ if for every position
$x' \in V'$   there is a move to a position $x'' \in V''$.


\begin{lemma} [\cite{GurHo15}] \label{Mis}
A game  $G$ is miserable if and only if there exist two disjoint sets  $V'_{0,1}$ and $V'_{1,0}$
satisfying the following five conditions:
\begin{enumerate}\itemsep0em
\item [\rm{(i)}]   Both sets  $V'_{0,1}$ and $V'_{1,0}$ are independent, that is, there is no legal move
between two positions of one set.
\item [\rm{(ii)}]  $V'_{0,1}$ contains all terminal positions, that is, $V_T \subseteq V'_{0,1}$.
\item [\rm{(iii)}] $ V'_{0,1} \setminus V_T$  is movable to $V'_{1,0}$.
\item [\rm{(iv)}]  $V'_{1,0}$  is movable to  $V'_{0,1}$.
\item [\rm{(v)}]   From every position $x \notin V'_{0,1} \cup V'_{1,0}$, either there is a move to $V'_{0,1}$ and
another one to $V'_{1,0}$ or there is no move to  $V'_{0,1} \cup V'_{1,0}$.
\end{enumerate}

Moreover,  $V'_{0,1} = V_{0,1}$ and $V'_{1,0} = V_{1,0}$  whenever the above five conditions hold.
\qed
\end{lemma}


{\bf Proof of Proposition \ref{p9}.}
For both cases  
$k=n$  and   
$k=n-1$, we will show that
the sets defined by this proposition satisfy
all conditions of Lemma \ref{Mis}, thus,
proving miserability and getting all swap positions.

We will consider only case  $k=n-1$, leaving the simpler one  $k = n$ to the reader.

Set
$V'_{0,1} = \{(i,i,\ldots,i,i+2j) | i, j \in \Z_{\geq}\}$
\newline
and
$V'_{1,0} = \{(i,i,\ldots,i,i+2j+1) | i, j \in \Z_{\geq}\}$.

Clearly, a move from a position $(i,i,\ldots,i,i+2j)$ reduces 
at most $n-1$ first or last coordinates and,
hence, cannot lead to a position $(i',i',\ldots,i',i'+2j')$.

Therefore, \rm{(i)} holds for $V'_{0,1}$.
Similarly, \rm{(i)} holds for $V'_{1,0}$.

Clearly, $V'_{0,1}$ contains the (unique) terminal position, with no tokens at all, and hence, \rm{(ii)} holds.

Conditions \rm{(iii)} and \rm{(iv)} hold, since moves
$(i,i,\ldots,i,i+2j)  \to (i,i,\ldots,i,i+2j-1)$ and
$(i,i,\ldots,i,i+2j+1) \to  (i,i,\ldots,i,i+2j)$, reducing the last coordinate, are legal.

Let us check \rm{(v)}.
Assume that the part
``there is no move from $x$  to  $V'_{0,1} \cup V'_{1,0}$"  fails and show that
there is a move from  $x$  to   $V'_{0,1}$ and another one to  $V'_{1,0}$.

Let us first assume that there is a move $\mathfrak{m}$ from $x$  to  $V'_{0,1}$  and
show that there is also a move from $x$ to  $V'_{1,0}$.
Consider two cases: 
\begin{enumerate} \itemsep0em
\item [(a)] Case 1: move $\mathfrak{m}$ reduces   
$x_1, x_2, \ldots, x_{n-1}$.  
In other words, $x = (i+k_1,i+k_2,\ldots, i+k_{n-1},i+2j)$ for some $i,j$ and
$0 \leq k_l \leq 1$ with $k_1 + k_2 + \cdots + k_{n-1} > 0$.
Note that $j \geq 1$, since there is a coordinate $i+k_j > i$.
If $k_l = 1$ for all $l$ then $x \in V'_{1,0}$, which is a contradiction.
Therefore, without loss of generality, we can assume that $k_1 = 0$.
In this case, move $x \to (i,i,\ldots, i,i+2j-1)$ terminates in $V'_{1,0}$.
\item [(b)] Case 2: move $\mathfrak{m}$ reduces 
the last coordinate  $x_n$. 
Without loss of generality, assume that this move 
reduces the last $n-1$
coordinates  $x_2, x_3, \ldots, x_n$, that is,
$x = (i,i+k_1,\ldots, i+k_{n-2},i+2j+k_{n-1})$ for some $i,j$ and
$0 \leq k_l \leq 1$ with $k_1 + k_2 + \cdots + k_{n-1} > 0$.
Note that $j \geq 1$ when $k_{n-1} = 0$ and also that
there is $l_0 \leq n-2$ such that $k_{l_0} = 1$, since otherwise $x \in V'_{0,1} \cup V'_{1,0}$.
Let us now consider $k_{n-1}$.
If $k_{n-1} = 0$  (resp.~$k_{n-1} = 1$) then
move $x \to (i,i,\ldots, i,i+2j-1)$ terminates in $V'_{1,0}$ (resp.~ in $V'_{1,0}$).
Note that the latter move is legal since $k_{l_0} = 1$.
\end{enumerate}

It remains to assume that there is a move from  $x$  to $V'_{1,0}$ .
and to show that there is a move from  $x$  to  $V'_{0,1}$ too.
The arguments are essentially the same as above and we leave this case to the reader.

Thus, by Lemma \ref{Mis}, the game is miserable;
moreover, $V'_{0,1}$  and $V'_{1,0}$
are its swap positions, that is,  $V_{0,1} = V'_{0,1}$  and $V_{1,0} = V'_{1,0}$.
\qed

\medskip

Our computations show that
{\sc Nim}$^1_{4, \leq 2}$ is not tame:
for example, $(1, 1, 2, 3)$  is a  $(4, 0)$-position.
Unlike the normal version, for the mis\`ere one
we have no simple characterization even for the  $\P$-positions.


\section{{\sc Slow exact $k$-{\sc Nim}} ({\sc Nim}$^1_{n, =k}$)}  \label{S.Exact}
As we already mentioned, the game is trivial when  $k=1$ or $k=n$.
In both cases there are only  $(0,1)$- and $(1,0)$-positions, which
alternate with every  move.
Whether  $x$  is a $(0,1)$- or a $(1,0)$-position depends on only the parity.
More precisely, $\G(x) = q(x) \pmod 2$, where
$q(x) = \sum_{i=1}^n x_i$  when  $k=1$  and
$q(x) = m(x) = \min_{i=1}^n x_i = x_1$  when  $k=n$;
in both cases  $\G^-(x) + \G(x) = 1$.

We will show that both the mis\`ere and normal versions of the game are tractable when  $n=3$ and $k=2$.
Again the parity vector plays an important role, although
in this case it does not define the SG function uniquely.

\subsection{On Sprague-Grundy function of {\sc Nim}$^1_{3,=2}$}
\begin{proposition} \label{p10}

Let us set
\begin{align*}
A   &= \{(2a, 2b-1, 2(b+i)) |     0 \leq a < b,    0 \leq i < a, (a+i)   \bmod{2} = 1\}, \\
B   &= \{(2a,2b,2(b+i)+1) |       0 \leq a \leq b, 0 \leq i < a, (a+i)   \bmod{2} = 1\}, \\
C_0 &= \{(2a-1, 2b-1, 2(b+i)-1) | 0 \leq a \leq b, 0 \leq i < a, (a+i)   \bmod{2} = 0\}, \\
C_1 &= \{(2a-1, 2b-1, 2(b+i)-1) | 0 \leq a \leq b, 0 \leq i < a, (a+i)   \bmod{2} = 1\}, \\
D_0 &= \{(2a-1, 2b, 2(b+i)) |     0 \leq a < b,    0 \leq i < a, (a+i) \bmod{2} = 1\}, \\
D_1 &= \{(2a-1, 2b, 2(b+i)) |     0 \leq a < b,    0 \leq i < a, (a+i) \bmod{2} = 0\}, \\
C   &= C_1 \cup C_2, D  = D_1 \cup D_2,
\end{align*}

The SG function of game {\sc Nim}$^1_{3,=2}$ takes only four values, $0,1,2,3$
and the sets of its  $i$-positions are as follows:

\begin{align*}
S_0 &= (\{(2a,2b,c) \mid 2a \leq 2b \leq c\} \setminus B) \cup A \cup C_0 \cup D_0, \\
S_1 &= (\{(2a,2b+1,c) \mid 2a \leq 2b+1 \leq c\} \setminus A) \cup B \cup C_1 \cup D_1, \\
S_2 &= \{(2a+1,2b+1,c) \mid 2a+1 \leq 2b+1 \leq c\} \setminus C, \\
S_3 &= \{(2a+1,2b,c) \mid 2a+1 \leq 2b \leq c\} \setminus D.
\end{align*}
\end{proposition}

\begin{proof}
It is easily seen that these four sets
partition the set of all positions of {\sc Nim}$^1_{3,=2}$
In addition, we will verify the following four statements that
immediately imply that  $S_i$  is the set of  $i$-positions, by definition of SG function.

\begin{enumerate} \itemsep0em
\item each set  $S_i$ is independent, that is, there is no move between any two of its positions;
\item if $p \notin S_0$ then $p$ is movable to $S_0$;
\item if $p \notin S_0 \cup S_1$ then $p$ is movable to $S_1$;
\item if $p \notin S_0 \cup S_1 \cup S_2$ then $p$ is movable to $S_2$.
\end{enumerate}

\begin{enumerate} \itemsep0em
\item Case $i = 0$.
Let  $p \in S_0$.
Assume that  $p \in (\{(2a,2b,c) \mid 2a \leq 2b \leq c\} \setminus B)$.
There are three types of moves from $p$.
\begin{enumerate} \itemsep0em
\item [\rm{(i)}] $p \to (2a-1,2b-1,c) = q$.
By the parity vector of $q$, we have  $q \notin \{(2a,2b,c) \mid 2a \leq 2b \leq c\} \cup A \cup D_0$.
Since $p \notin B$, one can derive that  $q \notin C_0$.
Indeed, assume that $q \in C_0$.
Then $c = 2(b+i)-1$ for some $i$ such that $0 \leq i < a$ and $(a+i) \bmod{2} = 0$.
Note that $i \geq 1$, since $c \geq 2b$. We have $p = (2a, 2b, 2(b+i)-1) = (2a,2b,2(b+i-1)+1)$, where
$0 \leq i-1 < a$ and $(a+i-1) \bmod{2} = 1$.
The last formula for $p$ implies that $p \in B$, resulting in a contradiction.
Thus, $q \notin C_0$.
It follows that $q \notin S_0$.
\item [\rm{(ii)}] $p \to (2a-1,2b,c-1) = q$.
If $c-1 = 2b-1$, $q = (2a-1,2b-1,2b) \notin S_0$ as shown in the case \rm{(i)}.
If $c-1 \geq 2b$  then  $q \notin \{(2a,2b,c) \mid 2a \leq 2b \leq c\} \cup A \cup C_0$.
Since $p \notin B$, one can conclude that $q \notin D_0$.
Therefore, $q \notin S_0$.
\item [\rm{(iii)}] $p \to (2a,2b-1,c-1) = q$.
If  $2a > c-1$ then  $p = (2a,2a,2a)$ and, hence, $q \notin S_0$, as shown in case \rm{(i)}.
Assume that $2a \leq c-1$.
If $2a > 2b-1$ then $p = (2a,2a,c)$ and, hence, $q \notin S_0$ as shown in case \rm{(ii)}.
Thus, we can assume that $2a \leq 2b-1$ and, hence, $2a < 2b-1$.
In other words, the first (smallest) coordinate of $q$ is $2a$ and the second is $2b-1$.
Therefore, $q \notin \{(2a,2b,c) \mid  2a \leq 2b \leq c\} \cup C_0 \cup D_0$.
Since $p \notin B$, one can verify that $q \notin A$.
Thus, $q \notin S_0$.
\end{enumerate}

Similarly, one can show that there is no move from $p$ to $S_0$, when $p \in A \cup C_0 \cup D_0$.

The case $i = 1$ is similar to the case $i = 0$ and we leave to the reader.

If $i = 2$  then a move from a position $(2a+1,2b+1,c) \in S_2$  results
in a position with the first or the second even coordinate.
Clearly, such a move cannot terminate in $S_2$.

If $i = 3$  then a move from $(2a+1,2b,c) \in S_3$ will change the parity of
either the first or the second coordinate without changing any order. 
Clearly, such a move cannot terminate in $S_3$.
\item Let $p \notin S_0$.
Consider the following four cases: 

\begin{enumerate} \itemsep0em
\item [\rm{(i)}]   $p = (2a,2b,c)$  such that  $2a \leq 2b \leq c$.
Then, $p \in B$, that is, $p = (2a,2b,2(b+i)+1)$ for some $i$ such that $0 \leq g <  a$ and $(a+i) \bmod 2 = 1$.
If $a = b$ then move $p \to (2a-1,2b,2(b+i))$ terminates in $D_0 \subset S_0$;
if $a < b$ then move $p \to (2a,2b-1,2(b+i))$ terminates in $A \subset S_0$.
\item [\rm{(ii)}]  $p = (2a,2b+1,c)$ with $2a \leq 2b+1 \leq c$.
Consider the move $p \to (2a,2b,c-1) = q$.
Let us show that $p \notin A$ implies that $q \notin B$.
Assume that $q \in B$.
Then, $c - 1 = 2(b+i)+1$ for some $i$ such that $0 \leq i < a$ and $(a+i) \bmod{2} = 1$.
We set  $c = 2(b+1+i) = 2(b'+i)$ with $b' = b+1$ and conclude that $p = (2a,2b'-1,2(b'+i))$.
By the restrictions  for  of $i$, we have $p \in A$, which is a contradiction.
Thus, $q \notin B$, implying also that $q \in S_0$.
\item [\rm{(iii)}] $p = (2a+1,2b+1,c)$ with $2a+1 \leq 2b+1 \leq c$.
Consider a move $p \to (2a,2b,c) = q$.
Since $p \notin C_0$, similarly to the argument of \rm{(ii)}, we prove that $q \notin B$ and, hence, $q \in S_0$.
\item [\rm{(iv)}]  $p = (2a+1,2b,c)$ such that $2a+1 \leq 2b \leq c$.
If $c = 2(b+i)+1$ then move $p \to (2a,2b,2(b+i))$  terminates in $S_0$.
Let  $c = 2(b+i)$. 
If $i \geq a+1$, consider move $p \to (2a,2b,2(b+i)-1) = (2a,2b,2(b+i')+1) = q$ such that  $i' = i-1$.
Since $i' \geq a$, we have $q \notin B$ and, hence, $q \in S_0$.
If $i \leq a$, set  $p = (2a'-1,2b,2(b+i))$, where  $0 \leq i < a' = a+1$.
We have $p \in D$.
Furthermore, since $p \notin S_0$, we have  $p \notin D_0$ and, hence,  $p \in D_1$,
implying that $(a' + i) \bmod{2} = 0$.
Considering move $p \to (2a'-1,2b-1,2(b+i)-1)$, we conclude that  $q \in C_0 \subset S_0$.
\end{enumerate}
\item  If  $p \notin S_0 \cup S_1$  then either $p \in S_2$ or $p \in S_3$. Consider both cases.
\begin{enumerate} \itemsep0em
\item [\rm{(i)}] $p \in S_2$.
Then $p = (2a+1,2b+1,c)$ for some $a,b,c$. We  examine the parity of $c$.
If $c$ is even or, equivalently, $c = 2(b+i)$, we consider move $p \to (2a, 2b+1,2(b+i)-1) = q$.
Note that $q \notin A$ since its third coordinate is odd. Hence, $q \in S_1$.
If $c$ is odd, since $p \notin C$, we have  $c = 2(b+i+1)-1$ for some $i \geq a+1$.
In fact, if $c = 2(b+i+1)-1$ for some $i < a+1$, defining $a'=a+1, b'=b+1$, we get  $p = (2a'-1,2b'-1,2(b'+i)-1) \in C$,
 which is a contradiction.
 Now consider move $p \to (2a,2b+1,2(b+i)) = (2a,2(b+1)-1,2(b+1+i-1)) = q$.
 Since $i-1 \geq a$, $q \notin A$ and, hence,  $q \in S_1$.
 \item [\rm{(ii)}] $p \in S_3$.
 Then $p = (2a+1,2b,c)$ for some $a,b,c$.
 Again, we examine $c$.
 If $c$ is odd or, equivalently, $c = 2(b+i)+1$, we consider the move $p \to (2a, 2b-1,2(b+i)+1) = q$.
 Note that $q \notin A$ since its third coordinate is odd.
 Hence, $q \in S_1$.
 If $c$ is even then $c=2(b+i)$ for some $i \geq a+1$.
 In fact, if $i < a+1$, by setting  $p = (2(a+1)-1,2b,2(b+i))$, we get  $p \in D$, which is a a contradiction.
 Now consider move $p \to (2a,2b-1,2(b+i)) = q$. Since $i > a$, we have $q \notin A$ and, hence, $q \in S_1$.
 \end{enumerate}

\item Let $p \notin S_0 \cup S_1 \cup S_2$.
Then $p \in S_3$ and, hence, $p = (2a+1,2b,c)$ for some $a,b,c$; yet,  $p \notin D$.
Examine $c$ again.
If $c$ is odd or, equivalently, $p = (2a+1,2b,2(b+i)+1)$, we consider move $p \to (2a+1,2b-1,2(b+i)) = q$.
Since the third coordinate of $q$ is even, $q \notin C$ and, hence, $q \in S_2$.
If $c$ is even, $c = 2(b+i)$ for some $i \geq a+1$.
In fact, if $i < a+1 = a'$ then  $p = (2a'-1,2b,2(b+i)) \in D$, which is a contradiction.
Now consider move $p \to (2a+1,2b-1,2(b+i)-1) = (2a'-1,2b-1,2(b+i)-1) = q$ such that $i \geq a+1 = a'$.
Note that $q \notin C$ and, hence, $q \in S_2$.
\end{enumerate}
\end{proof}
\begin{remark}
For  $n = k+1 = 4$, our computations indicate that the SG function still takes only four values
$\{0,1,2,3\}$, corresponding to the parity vectors
$\{(e,e,*,*), (e,o,*,*), (o,o,*,*), (o,e,*,*)\}$, respectively.
Yet, the structure of exceptions is more complicated than in case  $n=k+1=3$  and
can hardly be characterized by simple closed formulas, like in Proposition \ref{p10}.
\end{remark}
\begin{remark}
For  $n = k+1 = 3$ the SG function takes only values  $\{0,1,2,3\}$, because
in this case there are at most three moves from each position.
The corresponding number is ``typically" $4$ when  $n = k+1 = 4$, yet, no $4$-position was found.
However, larger SG values can be taken when  $n = k+1 > 4$.
For example, our computations show that
$(1, 2, 3, 3, x_5)$  for  $3 \leq x_5 \leq 7$
are  $5$-positions of the game {\sc Nim}$^1_{5, =4}$.
\end{remark}

\subsection{On the Sprague-Grundy function of mis\`{e}re {\sc Nim}$^1_{3,=2}$}
\begin{proposition}
\label{p11}
Let us set
\begin{align*}
A_1 &= \{(2a,2b-1,2b+4i)        \mid a = 2a', 0 \leq i < a'\}, \\
A_2 &= \{(2a,2b-1,2b+4i+2)      \mid a = 2a'+1, 0 \leq i < a'\}, \\
B_1 &= \{(2a,2b,2b+4i+1)        \mid a = 2a', 0 \leq i < a'\}, \\
B_2 &= \{(2a,2b,2b+4i+3)        \mid a = 2a'+1, 0 \leq i < a'\}, \\
C_0 &= \{(2a-1, 2b-1, 2(b+i)-1) \mid 0 \leq i < a-1, (a+i) \bmod{2} = 0\}, \\
C_1 &= \{(2a-1, 2b-1, 2(b+i)-1) \mid 0 \leq i < a-1, (a+i) \bmod{2} = 1\}, \\
D_0 &= \{(2a-1, 2b, 2(b+i))     \mid 0 \leq i < a-1, (a+i) \bmod{2} = 1\}, \\
D_1 &= \{(2a-1, 2b, 2(b+i))     \mid 0 \leq i < a-1, (a+i) \bmod{2} = 0\}, \\
E   &= \{1, 2b-1, 2b-1          \mid b \geq 1 \}, \\
F   &= \{1, 2b, 2b              \mid b \geq 1 \}, \\
A   &= A_1 \cup A_2, B = B_1 \cup B_2, C = C_1 \cup C_2, D  = D_1 \cup D_2,
\end{align*}

The SG function of the of mis\`{e}re version of {\sc Nim}$^1_{3,=2}$
takes only four values, $0,1,2,3$
and the sets of its  $i$-positions are as follows:

\begin{align*}
S_0 &= (\{(2a,2b,c)   \mid 2a \leq 2b \leq c\} \setminus B) \cup A \cup C_0 \cup D_0 \cup E, \\
S_1 &= (\{(2a,2b+1,c) \mid 2a \leq 2b+1 \leq c\} \setminus A) \cup B \cup C_1 \cup D_1 \cup F, \\
S_2 &= \{(2a+1,2b+1,c)\mid 2a+1 \leq 2b+1 \leq c\} \setminus (C \cup E), \\
S_3 &= \{(2a+1,2b,c)  \mid 2a+1 \leq 2b \leq c\} \setminus (D \cup F).
\end{align*}
\end{proposition}
\begin{proof}
It is essentially similar to the proof of Proposition \ref{p10} and we leave to the reader.
\end{proof}

\subsection{Swap, non-swap, and non-tame positions}
The relatively simple closed formulas
obtained in Propositions \ref{p9} and \ref{p10}, respectively,
for the normal and mis\`ere versions of {\sc Nim}$^1_{3, =2}$
allow us to characterize the swap positions of the game as follows.

\begin{corollary} \label{c3}
Consider position $x = (x_1, x_2, x_3)$ with $0 \leq x_1 \leq x_2 \leq x_3$.
\begin{enumerate} \itemsep0em
\item [$(1)$] If  $x_1 = 0$ then $x$  is a  $(0,1)$-position
$($respectively,~ $(1,0)$-position$)$ whenever $x_2$ is even $($resp.~odd$)$.
\item [$(2)$] If  $x_1 = 1$ then $x$  is a  $(0,1)$-position
$($respectively,~$(1,0)$-position$)$ whenever $x_2 = x_3$ is even $($resp.~$x_2 = x_3$ is odd$)$.
\item [$(3)$] For  $x_1 \geq 2$, let us set $x_1 = a, x_2 = b$, and $x_3 = b+i$.
Then,  $x = (a,b,b+i)$  is a swap position if and only if $i \pmod 2 \neq a \pmod 2$, where
$i < a$ when $a$ is even  and $i < a-2$ when $a$ is odd.

Furthermore, in this case $\G(x) = b + \lceil \frac {i}{2} \rceil$ (while $\G^-(x) = 1 - \G(x)$).

\end{enumerate}
\end{corollary}
\proof
It follows directly from Propositions \ref{p10} and \ref{p11}.
\qed

A position  $x$  is called {\em tame} if $\G(x) = \G^-(x)$.
Recall that a game is tame if and only if it has only tame and swap positions.
A game is called {\it domestic} \cite{GurHo15} if
it has neither $(0,k)$ nor $(k,0)$ position for  $k \geq 2$.
The following result shows that {\sc Nim$^1_{3,=2}$} is not domestic.
Moreover, we can characterize the positions that are neither tame nor swap.

\begin{corollary} \label{c4}
A position $x = (x_1, x_2, x_3)$  such that
$0 \leq x_1 \leq x_2 \leq x_3$ is neither swap nor tame if and only if
$x_1$ is odd, $x_1 \neq 1$, and  $x_1 + x_2 = x_3 + 1$.
Under these conditions, $x$ is a $(0,3)$-position $($resp.~$(1,2)$-position$)$
if and only if $x_2$ is even $($resp.~odd$)$.
\end{corollary}

\proof
It results directly from Propositions \ref{p9} and \ref{p10}.
\qed

\bibliographystyle{amsplain}

\end{document}